\documentclass[12pt,oneside,reqno]{amsart}

\usepackage{amssymb, amstext, amscd, amsmath, color}

\usepackage{url}

\usepackage[pdftex]{graphicx}
\usepackage{epstopdf}
\usepackage{amsfonts}
\usepackage{amsthm}
\usepackage{amsfonts}
\usepackage{array}
\usepackage{epsfig}
\usepackage{eucal}
\usepackage{latexsym}
\usepackage{mathrsfs}
\usepackage{textcomp}
\usepackage{verbatim}
\usepackage{setspace}
\newtheorem{thm}{Theorem}[section]

\newtheorem{cor}[thm]{Corollary}

\newtheorem{defn}[thm]{Definition}

\newtheorem{lem}[thm]{Lemma}

\numberwithin{equation}{section}

\newcommand{\bQ}{{\mathbb{Q}}}
\newcommand{\bR}{{\mathbb{R}}}

  \newcommand{\F}{{\mathcal{F}}}

\renewcommand{\L}{{\mathcal{L}}}
  \newcommand{\M}{{\mathcal{M}}}

  \newcommand{\R}{{\mathcal{R}}}
\renewcommand{\S}{{\mathcal{S}}}
  \newcommand{\T}{{\mathcal{T}}}
  \newcommand{\U}{{\mathcal{U}}}

\newcommand{\ul}{\underline  }
\newcommand{\rank}{\operatorname{rank}}
\newcommand{\nulty}{\operatorname{null}}
\begin{document}

%\hfill {\color{blue}  draft, August 21, 2012}

\title[Generically rigid frameworks on surfaces]{A characterisation of generically rigid frameworks on surfaces of revolution}
\author[A. Nixon]{A. Nixon*}
\address{Heilbronn Institute for Math. Research\\ School of Math.\\ University of Bristol\\ Bristol\\
BS8 1TW \\ U.K. }
\email{tony.nixon@bristol.ac.uk}
\thanks{*part of this research was carried out at the Fields Institute, Toronto}
\author[J.C. Owen]{J.C. Owen}
\address{D-Cubed, Siemens PLM Software, Francis House\\
112 Hills Road, Cambridge, CB2 1PH, U.K.}
\email{owen.john.ext@siemens.com}
\author[S.C. Power]{S.C.  Power$^\dagger$}
% \thanks{Third author partially supported by an SERC grant.}
\address{Dept.\ Math.\ Stats.\\ Lancaster University\\
Lancaster LA1 4YF \\U.K. }
\email{s.power@lancaster.ac.uk}
\thanks{$^\dagger$supported by EPSRC grant  EP/J008648/1}

\thanks{2010 {\it  Mathematics Subject Classification.}
52C25, 05B35, 05C10, 53A05 \\
Key words and phrases: bar-joint framework, infinitesimally rigid, framework on a surface}

\begin{abstract}A foundational theorem of
Laman  provides a counting characterisation of the finite simple graphs
whose generic bar-joint frameworks in two dimensions are  infinitesimally rigid.
Recently a Laman-type characterisation was obtained for  frameworks in three dimensions whose vertices are constrained to concentric spheres or to concentric cylinders. Noting that the plane and the sphere have $3$ independent locally tangential infinitesimal motions while the cylinder has $2$, we obtain here
a Laman-type theorem for frameworks on algebraic surfaces with a $1$-dimensional space of tangential motions. Such surfaces include the torus, helicoids and surfaces of revolution. The relevant class of graphs are the $(2,1)$-tight graphs, in contrast to $(2,3)$-tightness for the plane/sphere and $(2,2)$-tightness for the cylinder. The proof uses a new characterisation of simple $(2,1)$-tight graphs and an inductive construction requiring generic rigidity preservation for $5$ graph moves, including the two Henneberg moves, an edge joining move and various vertex surgery moves.
\end{abstract}

\date{}
\maketitle

\section{Introduction}

A bar-joint framework in real Euclidean space $\bR^d$ is a geometric realisation of
the vertices of a graph with the edges considered as inextensible bars between them.
Such a framework is said to be rigid if there is no non-trivial continuous motion of the framework vertices
which maintains bar-lengths, and is said to be flexible if it is not rigid.
A foundational theorem of
Laman, obtained in 1970,  asserts that the rigidity of a generically positioned framework in the plane depends only on the underlying
graph and furthermore these graphs are characterised in terms of a simple counting condition.  There is also an elegant recursive construction of the
minimally rigid frameworks
going back to Henneberg \cite{Hen,Lam} in which each framework may be derived
from a single edge framework by repeated application of two simple construction moves,
namely the Henneberg 1 move and the Henneberg 2 move.
Analogous characterisations for frameworks in $\bR^3$ remain open problems  and
no combinatorial characterisation of generic rigidity is known.
We note however that a number of partial and related results are given in
Whiteley \cite{whi-book} and that
the longstanding molecular conjecture has been resolved by Katoh and Tanigawa \cite{K&T}.

Attention has also been given to frameworks in $3$-dimensional space whose vertices are constrained to  $2$-dimensional surfaces. In \cite{NOP} we obtained Laman-Henneberg-type theorems for the case of constraint to parallel planes,
concentric spheres and concentric cylinders.
The cylinder case presents  new complications both for the purely graph theoretical analysis and for the preservation of rigidity under the Henneberg moves and the further construction moves that are needed.
%, namely the vertex-to-$K_4$ move and the vertex-splitting move.

A fixed plane surface in $\bR^3$ supports a three-dimensional vector space of {internal} infinitesimal motions, coming from translations and rotations, while the (infinite circular) cylinder has two such independent motions.
More formally, in Definition \ref{d:surfacetype} we define the type $k$
of an irreducible algebraic surface where $k$ takes values
$3,2,1$ or $0$. (See also Definition  \ref{d:surfacetype2}.)
The type of a surface is reflected in the
graph counting conditions for Laman type theorems; fewer independent infinitesimal motions for the surface
imply a richer set of graphs which in turn require more constructive moves and more refined
proof techniques. 

The main result in the present paper is the following theorem for
bar-joint frameworks in  $\bR^3$ whose vertices are constrained to
an algebraic surface $\M$ of type $1$. 
These surfaces include a range of fundamental algebraic surfaces such as the elliptical cylinder, the cone, the torus,  surfaces of revolution, and various helical glide-translation surfaces.

\begin{thm}\label{conetorustheorem}
Let $G=(V,E)$ be a simple graph, let $\M$ be an irreducible algebraic surface in $\bR^3$ of type $1$ and let $(G,p)$ be a generic framework on $\M$.
Then $(G,p)$ is isostatic on
$\M$ if and only if
%$|E|=2|V|-1$ and for every subgraph $H=(V',E')$, $|E'|\leq 2|V'|-1$.
$G$ is $K_1, K_2, K_3, K_4$ or is $(2,1)$-tight.
\end{thm}

For a plane surface the required graphs are the Laman graphs, corresponding to
the top count $|E|=2|V|-3$
and the inequality $|E'|\leq2|V'|-3$ for every subgraph $(V', E')$.
In fact these are the $(2,3)$-tight graphs, in the sense of Definition \ref{kltight} and are necessarily simple. Observe that these graphs are necessarily simple. For the cylinder the appropriate graphs are the simple $(2,2)$-tight graphs. In this paper the key class of graphs are the simple $(2,1)$-tight graphs which were characterised recently in
Nixon and Owen \cite{N&O}. We shall obtain here the following alternative characterisation which turns out to be efficient for our purposes. The methods for this also lead to a new analogous characterisation of simple $(2,2)$-tight graphs given in Section \ref{graphtheory}.

\begin{thm}\label{t:21characterisation} 
A simple finite graph is $(2,1)$-tight if and only if it is equal to $K_5\backslash e$ or can be obtained from this graph by the sequential application of moves of 5 types, namely the Henneberg $1$ and $2$ moves, the vertex-to-$K_4$ move, the vertex-to-$4$-cycle move, and the edge joining move.
\end{thm}

The proof of the main theorem is principally concerned with the sufficiency for rigidity of the combinatorial condition
and we obtain  this by showing that each of the moves in the constructive sequence for the graph
preserves generic rigidity. We introduce some new methods for this and
there are two moves that present particular challenges, namely the Henneberg $2$ move and the vertex-to-$K_4$ move. 
For the Henneberg $2$ move we give two quite different proofs which also give new proofs in the case of the circular cylinder.
The first of these uses a convergence argument involving a sequence of 
generic realisations of $G'$ 
which converge to a degenerate non-generic realisation of $G'$ which in a natural sense covers
a generic realisation of $G$. 
On the other hand in Section \ref{sec:algh2} we adopt the traditional approach
of algebraic specialisation to obtain a direct entirely algebraic proof. 

For the  vertex-to-$K_4$ move we show that proper flexes are inherited under the inverse move $G'\to G$ corresponding to $K_4$ contraction to a vertex. This is achieved through the
consideration of a sequence $(G', p^k)$ in which the $K_4$-subframeworks contract in a manner
which is \textit{well-behaved} with respect to the distinct principal curvatures of the surface.

In all our considerations  the framework vertices are constrained to a surface while the edges are straight Euclidean edges measured by distances in $\bR^3$. We note that Whiteley \cite{whi-union} has considered the topic of frameworks
on surfaces with geodesic edges and there the appropriate combinatorial objects are looped multigraphs. Moreover,  rigidity is expressed in terms
of the rank of the $k$-frame matrix rather than the rigidity matrix.
We note also that  frameworks on surfaces are also implicit in the context of periodic frameworks, considered, for example, by Borcea and Streinu \cite{bor-str}, Malestein and Theran \cite{M&T,M&Tcone}, Nixon and Ross \cite{nix-rosperiodic}, Owen and Power \cite{O&P}  and Ross \cite{Ross2}.

The paper is structured as follows. In Section \ref{surfaceintro} we recall  basic definitions and key results from \cite{NOP} for generic frameworks on surfaces.
In Section \ref{graphtheory} we detail the inductive moves on graphs and obtain inductive characterisations
of simple $(2,k)$-tight graphs for $k = 1$ and $2$.
In Section \ref{s:hennebergmoves} we prove the preservation of generic independence for the two Henneberg moves as moves on frameworks on algebraic surfaces. In Section  \ref{vertexmoves} we prove the preservation of generic independence for
various vertex surgery moves, including vertex splitting, vertex-to-$4$-cycle and vertex-to-$K_4$. In Section \ref{sec:algh2} we provide an alternative proof of generic independence under Henneberg 2 moves that we believe could be of independent interest.
In Section \ref{theoremsec} we prove the main theorem and in the final section we discuss the difficulty of extending our results to other surfaces and comment on higher dimensional contexts.

%Also, as an application we characterise generically isostatic frameworks
%which are supported on some two-dimensional surfaces in  four dimensional %Euclidean space, namely
%the Clifford torus $S^1 \times S^1$ in $\bR^2 \times \bR^2$ and its deformed %variant
%$S^1 \times E^1$, with $E^1$ a proper ellipse.

\section{Frameworks on Surfaces}
\label{surfaceintro}

Let $\M$ be a subset of  $\bR^3$ which is a smooth surface in the sense of being a $2$-dimensional embedded differentiable manifold.
The main examples we have in mind are defined as disjoint unions of
parts of elementary algebraic surfaces. Accordingly we assume  smoothness in the sense that local coordinate maps exist for
$\M$ which are analytic. In particular for every point  of $\M$ there is a continuous
choice of normal vectors in a neighbourhood of the point and a Taylor series expansion, as in Equation \ref{taylor}, for points of $\M$ in this neighbourhood.

A framework $(G,p)$ on $\M$ is a finite bar-joint framework in $\bR^3$, for a
simple graph $G=(V,E)$, with framework points $p(v), v\in V$, which lie on $\M$.

An \textit{infinitesimal flex} of $(G, p)$ on $\M$ is a sequence or vector $u$ of velocity vectors $u_1, \dots , u_{|V|}$,
considered as acting at the framework points, which  are tangential to the surface and
satisfy the infinitesimal flex requirement in $\bR^3$, namely
\[
 u_i.(p_i-p_j) = u_j.(p_i-p_j),
\]
for each edge $v_iv_j$.
It is elementary to show that $u$ is an infinitesimal flex if and only if $u$ lies in the
nullspace (kernel) of the rigidity matrix $R_\M(G,p)$ given in the following definition.
The submatrix of $R_{\M}(G,p)$ given by the first
$|E|$ rows provides the usual rigidity matrix, $R_3(G,p)$ say, for the unrestricted framework $(G,p)$.
The tangentiality condition corresponds to
$u$ lying in the nullspace of the matrix formed by the last $|V|$ rows.

\begin{defn}\label{rigiditymatrixdef}
The
\emph{rigidity matrix} $R_\M(G,p)$ of $(G,p)$ on $\M$ is an $|E|+|V|$ by $3|V|$ matrix in which consecutive triples of columns  correspond to
framework points. The first $|E|$ rows correspond to the edges and the entries in row $e=uv$ are zero except possibly in the column triples for
$p(u)$ and $p(v)$, where the entries are the coordinates of $p(u)-p(v)$ and $p(v)-p(u)$ respectively. The final $|V|$ rows correspond to the vertices and the entries
in the row for vertex $v$ are zero except in the columns for $v$ where the entries are the coordinates of a normal vector $N(p(v))$  to $\M$
at $p(v)$.
\end{defn}

The case of a surface $\M$ which is a subset of the nonsingular points of  a
polynomial equation $m(x,y,z)=0$  is of particular interest, especially when $m(x,y,z)$ is irreducible over some coefficient field. In what follows we confine attention
to the rational field and refer to such surfaces simply as \emph{irreducible surfaces}.
In this case we may take the derivative of $m(x,y,z)$ at $p(v)$ for the choice of normal $N(p(v))$. Furthermore, the rigidity matrix arises from the derivative of the augmented edge-function $\tilde{f}_G$, with
\[ 2R_{\M}(G,p) = (D\tilde{f}_G)(p), \]
where $\tilde{f}_G:\bR^{3|V|} \to \bR^{|E|+|V|}$ is given by $\tilde{f}_G(q) = (f_G(q), m(q_1), \dots , m(q_{|V|}))$
with
\[
f_G(q) =(\|q_i-q_j\|^2)_{v_iv_j\in E}
\]
the usual edge function for $G$ associated with
framework realisations in $\bR^{3n}$, where $\|.\|$ is the usual Euclidean norm.

As is well-known, for $n\geq 4$ a complete graph framework $(K_n,p)$ in $\bR^3$, not lying in a hyperplane, has a $6$-dimensional vector space of infinitesimal flexes,
%$u=(u_1,u_2, u_3), u_i \in \bR^3$.
a basis for which may be provided by a set of linearly independent infinitesimal flexes associated with translations  and rotations. When  the vertices of $(K_n,p)$ are constrained to a surface $\M$ then the dimension is reduced to $\dim \ker R_{\M}(K_n,p) = k$ where $k=3,2,1$ or $0$.
%We refer to $k$ as the \textit{number of trivial motions of $\M$ for $p$}.

We now define smooth surfaces of type $k$ for $k=3,2,1, 0$. The type number reflects the number of independent
infinitesimal motions of a typical framework on $\M$ that arise from isometries of $\bR^3$
that act tangentially at every point on $\M$ (not just the framework joints). For the sphere, cylinder and  cone the types are  $3, 2$ and $1$ respectively, while
the ellipsoid, defined by $x^2+2y^2+3z^2=1$, has type $0$.

\begin{defn}\label{d:surfacetype}
A surface $\M$ is said to be of type $k$, or to have freedom number $k$, if $\dim \ker R_{\M}(K_n,p)\geq k$ for all complete graph frameworks $(K_n,p)$ on $\M$ and $k$ is the largest such number.
\end{defn}

Apart from the type $3$ surfaces, which arise from concentric spheres or parallel planes,  a typical $K_4$ framework on a surface $\M$ has a two-dimensional space of infinitesimal flexes. This follows on consideration
of the $10$ by $12$ rigidity matrix.
For $K_3$ and $K_2$ frameworks the space is  three-dimensional and includes rotational flexes not derivable from (tangentially acting) isometries.
For the cylinder the flexes of $K_4$ frameworks are all associated with isometries whereas on the cone (resp. ellipsoid) there is a one-dimensional (resp. zero-dimensional) subspace determined by tangential isometries.

\begin{defn}\label{d:infrigidity} 
Let $\M$ be a smooth surface and $p=(p_1,\dots ,p_n)$  a vector
of points on $\M$. Then the framework $(G,p)$ on $\M$ is said to be infinitesimally rigid
if every infinitesimal flex of $(G,p)$ on $\M$ corresponds to a rigid motion flex of $\M$.
%In particular if
%$\dim \ker R_{\M}(K_{n},p)$ agrees with the freedom number $k$ of $\M$ then
%$(G,p)$   is infinitesimally rigid on $\M$ if and only if
%\[
%\dim \ker R_{\M}(G,p) = k.
%\]
\end{defn}

A framework $(G,p)$ on $\M$ is \emph{independent} if $R_\M(G,p)$ has linearly independent rows and is
\emph{minimally infinitesimally  rigid on $\M$}, or \emph{isostatic on $\M$} if it is independent and infinitesimally rigid on $\M$.

From the point of view of the infinitesimal rigidity  it is only the nature of $\M^{|V|}$ in a neighbourhood of $p$ which is of significance. On the other hand for irreducible surfaces one can establish generic properties for the pair $G, \M$ as we shall see.

Following Asimow and Roth we say that a framework $(G,p)$ on $\M$ is \emph{regular} if the rank of $R_\M(G,q)$ takes its maximum value throughout a neighbourhood of $p$ in $\M^{|V|}$.
In the case that $\M$ is an algebraic surface determined by an irreducible polynomial $m(x,y,z)$ over $\bQ$, the framework $(G, p)$ is
said to be \emph{generic} on $\M$ 
if an algebraic dependency $h(\{x_i\},\{y_i\},\{z_i\})=0$ holds between the coordinates $x_i, y_i, z_i$ of all the
points $p_i$ only when the polynomial $h(\{X_i\},\{Y_i\},\{Z_i\})$ lies in the ideal generated by the polynomials $m(X_i,Y_i,Z_i), 1 \leq i\leq |V|$. It is a standard exercise to show that such generic frameworks are regular.

Note that in contrast to type $3$ and $2$ there are diverse classical surfaces of type $1$,
including spheroids, with isometry group $S^1$, elliptical cylinders  and other noncircular cylinders,
with translational isometry group $\bR^1$, and circular hyperboloids and other diverse surfaces
with glide-rotation isometry group $\bR^1$.

\begin{defn}
A simple graph $G$ is independent for the irreducible surface $\M$ if every generic framework
$(G, p)$ on $\M$ is independent.
\end{defn}

In particular, $K_4$ is dependent for the sphere but independent for the cylinder. On the other hand
$K_5\backslash e$ is dependent for the cylinder but independent for the cone. Note that $K_n \backslash e$ denotes the unique graph formed by deleting any single
edge from $K_n$. 

The determination of combinatorial conditions for the generic independence of classes of frameworks is one of the fundamental problems in constraint system rigidity theory.
See for example  Whiteley \cite{whi-vertexsplit,whi-book} and Jackson and Jordan \cite{J&J}.
Our main result can be viewed in this spirit. Also we note that there is the following matroidal interpretation of our main result. Let
$\L(K_n,\M)$
be the linear matroid for the rigidity matrix
$R_\M(K_n,p)$  associated with a generic $n$-tuple and the irreducible surface $\M$.
Then by Theorem \ref{conetorustheorem} the bases of $\L(K_n,\M)$ correspond to sets of rows determined by the $(2,1)$-tight subgraphs.

The notions of continuous rigidity and minimal continuous rigidity are also naturally defined in the surface setting and the following equivalence is an analogue of a theorem of Gluck \cite{glu}.

\begin{thm}\cite{NOP}
A generic framework $(G,p)$ on an algebraic surface $\M$ is infinitesimally rigid if and
only if it is continuously rigid on $\M$.
\end{thm}

In this paper the infinitesimal rigidity perspective will be more direct
and we make use of the following two results from \cite{NOP}, namely a version of the Maxwell counting condition
and an isostatic characterisation in the spirit of  Asimow and Roth \cite{A&R}.

\begin{thm}\label{necessity}\cite{NOP}
Let $(G, p)$ be an isostatic generic framework on the algebraic surface $\M$  of type $k, 0\leq k\leq 3$,
with $G$ not equal to $K_1, K_2, K_3$ or $K_4$.
Then $|E|=2|V|-k$ and for every subgraph $H$ of $G$ with at least one edge,
$|E(H)|\leq 2|V(H)|-k$.
\end{thm}

\begin{thm}\label{isostaticnecessity}\cite{NOP}
%[Nixon, Owen and Power \cite{NOP}]
Let $(G,p)$ be a generic framework on a surface $\M$ of type $k$. Then $(G,p)$ is
isostatic on $\M$ if and only if
\begin{enumerate}
\item $\rank R_{\M}(G,p) = 3|V|-k$ and
\item $ 2|V|-|E| =k.$
\end{enumerate}
\end{thm}

The classes of graphs in Theorem \ref{necessity} are the simple graphs that are $(2,k)$-tight,
with $k = 0,1,2,3$, in the following sense.

\begin{defn}\label{kltight}
A graph $G = (V,E)$ is \emph{$(2,k)$-sparse} if for all subgraphs  $H$, with at least one edge,
the inequalities
$|E(H)| \leq 2|V(H)| - k$ hold.
Moreover $G$ is \emph{$(2,k)$-tight} if $G$ is $(2,k)$-sparse and $|E| = 2|V| - k$.
\end{defn}

The inductive characterisations of these classes of simple graphs for $k=3,2$ and $1$ form
a key part of our approach to proving the sufficiency of the necessary counting conditions for generic infinitesimal rigidity. We describe the various construction moves in the next section.

\begin{thm}[Henneberg \cite{Hen} and Laman \cite{Lam}]\label{23}
A simple graph $G$ is $(2,3)$-tight if and only if it can be generated from $K_2$ by Henneberg $1$ and $2$ moves.
\end{thm}

This characterisation plays a role in the following extension of Laman's theorem.

\begin{thm}\cite{NOP}\label{unionspheres}
Let $G=(V,E)$, let $\M$ be a union of parallel planes or a union of concentric spheres and let $p$ be generic
on $\M$. Then $(G,p)$ is isostatic on $\M$ if and only if $G$ is $K_1, K_2$ or $(2,3)$-tight.
\end{thm}

\begin{thm}\cite{NOP,N&O}\label{22}
For a simple graph $G$ the following are equivalent:
\begin{enumerate}
\item $G$ is $(2,2)$-tight,
\item $G$ can be generated from $K_1$ by Henneberg $1$, Henneberg $2$ and graph extension moves,
\item $G$ can be generated from $K_1$ by Henneberg $1$, Henneberg $2$, vertex-to-$K_4$ and vertex splitting moves.
\end{enumerate}
\end{thm}

\begin{thm}\cite{NOP}\label{unioncylinders}
Let $G=(V,E)$, let $\M$ be a cylinder or a union of concentric cylinders and let $p$ be generic on $\M$. Then $(G,p)$ is isostatic on
$\M$ if and only if $G$ is  $K_1, K_2, K_3$ or is $(2,2)$-tight.
\end{thm}

We observe that this theorem can also be proven, using the methods of this paper, by applying the equivalence of $(1)$ and $(3)$ in Theorem \ref{22} or by applying Theorem \ref{t:22characterisation}.

%Let $K_n \sqcup K_m$ denote the unique graph
%which is the join of $K_n$ and $K_m$ over a single common edge and its two vertices.
%The following theorem provides an alternative  to  Theorem \ref{t:21characterisation} with vertex splitting replacing the vertex-to-$4$-cycle move.

We also note the following alternative characterisation of $(2,1)$-tight simple graphs to that which is given in Theorem \ref{t:21characterisation}. This will not be needed for the proofs in this paper but we note that at the expense of introducing
the new base graph $K_4\sqcup K_4$ one may replace the vertex-to-$4$-cycle move by the more localised vertex splitting move. ($K_4 \sqcup K_4$ consists of two copies of $K_4$ sharing exactly one edge.)

\begin{thm}\cite{N&O}\label{21}
A simple graph $G$ is $(2,1)$-tight if and only if it can be generated from $K_5\backslash e$ or $K_4 \sqcup K_4$ by Henneberg $1$, Henneberg $2$, vertex-to-$K_4$, vertex splitting and edge joining moves.
\end{thm}

\section{Simple $(2,1)$-tight graphs}
\label{graphtheory}

In this section we prove Theorem \ref{t:21characterisation} and in Theorem \ref{t:22characterisation} we obtain an analogous
characterisation of $(2,2)$-tight simple graphs. We remark that the insistence on simplicity of the graph makes the construction problematic. Indeed if we permit loops and parallel edges then a general construction theorem of Fekete and Szeg\H{o} \cite{F&Z} applies. In the case of $(2,1)$-tight graphs it ensures that the only operations required are Henneberg 1 and 2 type operations. Fekete and Szeg\H{o}'s result extended work of Tay on $(k,k)$-tight graphs for the characterisation of the generic rigidity of body-bar frameworks in arbitrary dimension \cite{Tay2}.

The inductive characterisation of $(2,3)$-tight graphs of Henneberg and Laman
starts with the  elementary counting observation that for $k \geq 1$ the average degree in a
$(2,k)$-tight graph is less than $4$, while there can be no vertices of degree less than $2$.
Thus degree $2$ or degree $3$ vertices exist. Henneberg introduced the following two operations on graphs
which maintain the sparsity count, increase the vertex count by $1$, and add a new vertex of
degree $2$ or degree $3$ respectively.
See also the discussions in \cite{Lam,nix-ros,T&W}.

The \emph{Henneberg $1$} move augments a graph $G$ by adding
a vertex $v$ of degree $2$  and two edges $vv_1, vv_2$ from it to distinct neighbours
$v_1, v_2$ in $G$.

The \emph{Henneberg $2$} move  removes an edge $v_1v_2$ from a graph and adds a vertex $v$ of degree
$3$ with distinct neighbours $v_1, v_2, v_3$ for some vertex $v_3$. 
%see Figure \ref{hen2}.
This is also referred to as edge splitting or 1-extension in the literature.

\begin{center}
\begin{figure}[ht]
\centering
\includegraphics[width=9cm]{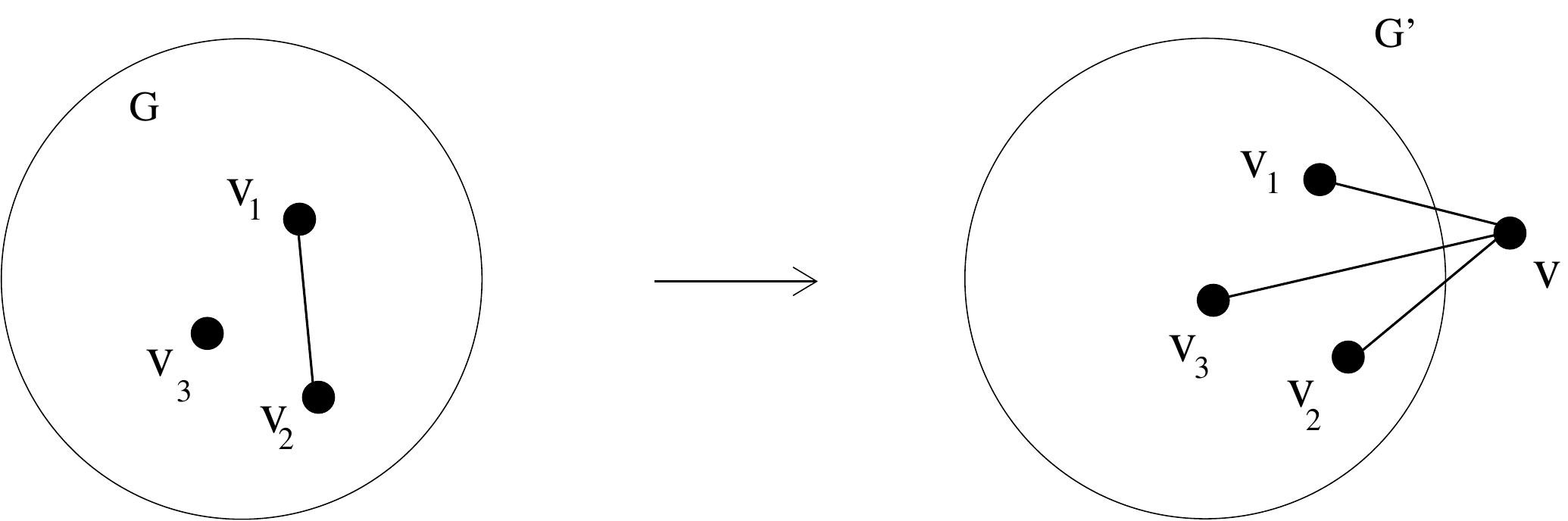}
\caption{The Henneberg $2$ move.}
\label{hen2}
\end{figure}
\end{center}

Laman proved that every $(2,3)$-tight graph can be generated recursively from $K_2$ by a sequence of these operations. The key step is to show that an  inverse Henneberg $2$ operation
is possible on a degree 3 vertex $v$, by virtue of the fact that at least one of the three choices for the new edge ($v_1v_2, v_2v_3$ or $v_3v_1$) can be made without violating any subgraph count.

On the other hand, for a $(2,2)$-tight graph it is easy to see that degree $3$ vertices may be contained within subgraphs isomorphic to $K_4$ and so admit no inverse Henneberg 2 move. Indeed there are countably many $(2,2)$-tight graphs for which every vertex of degree less than $4$ is contained in a copy of $K_4$.
In view of this obstruction, in \cite{NOP,N&O} we considered additional graph moves preserving $(2,2)$-tightness, including those indicated in Figure \ref{Extension} and Figure \ref{Vertex splitting}.
These are the  \emph{vertex-to-$K_4$ move} and the \emph{vertex splitting move}.

\begin{center}
\begin{figure}[ht]
\centering
\includegraphics[width=6cm]{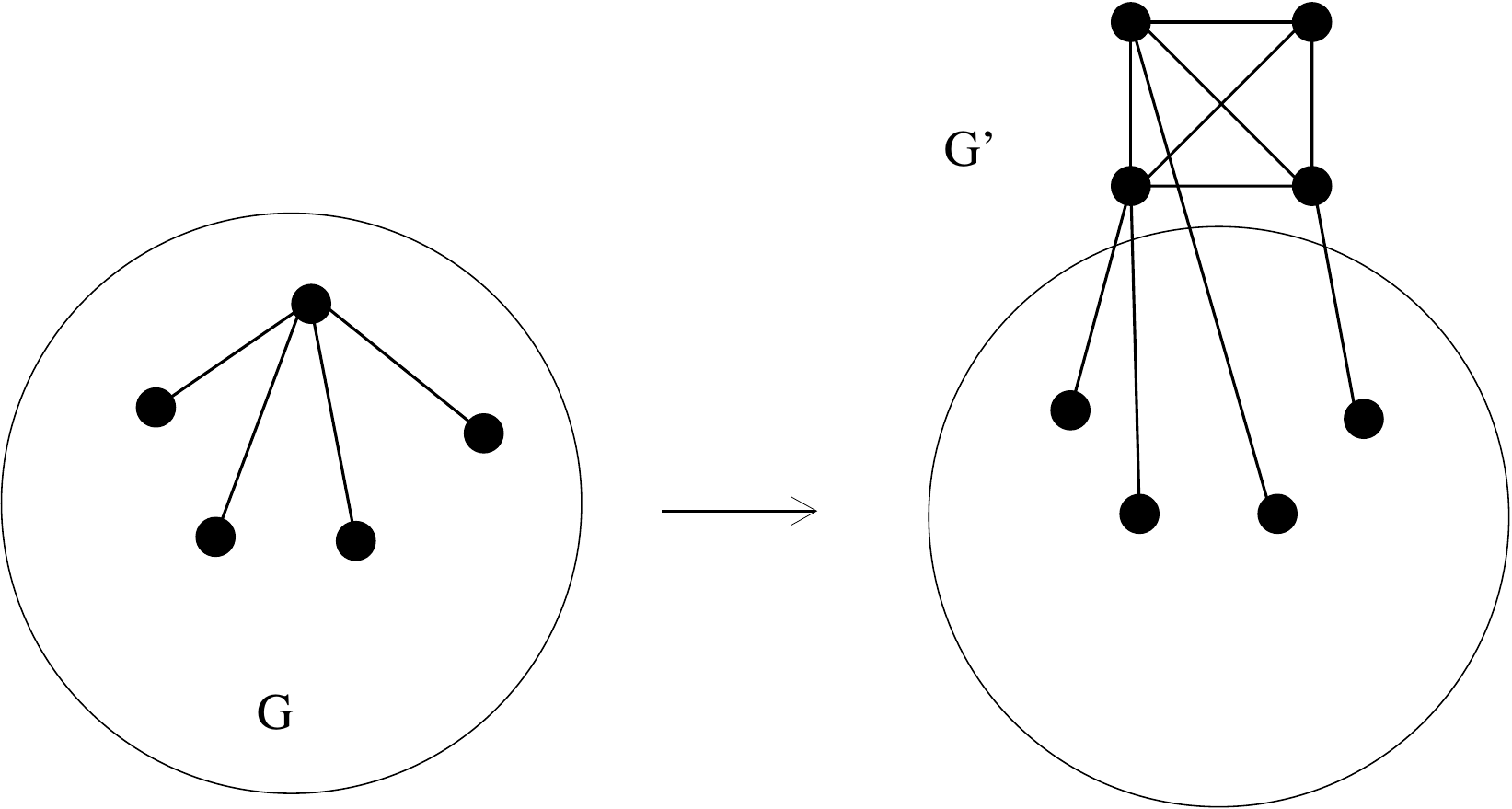}
\caption{The vertex-to-$K_4$ move.}
\label{Extension}
\end{figure}
\end{center}

The \emph{vertex-to-$K_4$ move} substitutes a copy of $K_4$ in place of a vertex $v$, with an arbitrary replacement of edges $xv$ by edges $xw$ with $w$ in $V(K_4)$. %More generally, the \textit{graph extension move} performs similar surgery with $v$ replaced by a $(2,2)$-tight graph. 
The inverse operation, contracting a copy of $K_4$ to a single vertex will be called the $K_4$-to-vertex move.

%The inverse move associated with the vertex-to-$K_4$ move is not always admissible, in the sense that, it does not always preserve the property of being $(2,k)$-tight. When the inverse move is admissible we call it an admissible $K_4$-to-vertex move.

%we refer to as an admissible $K_4$-to-vertex move, or simply as an admissible $K_4$ contraction.

\begin{center}
\begin{figure}[ht]
\centering
\includegraphics[width=6.6cm]{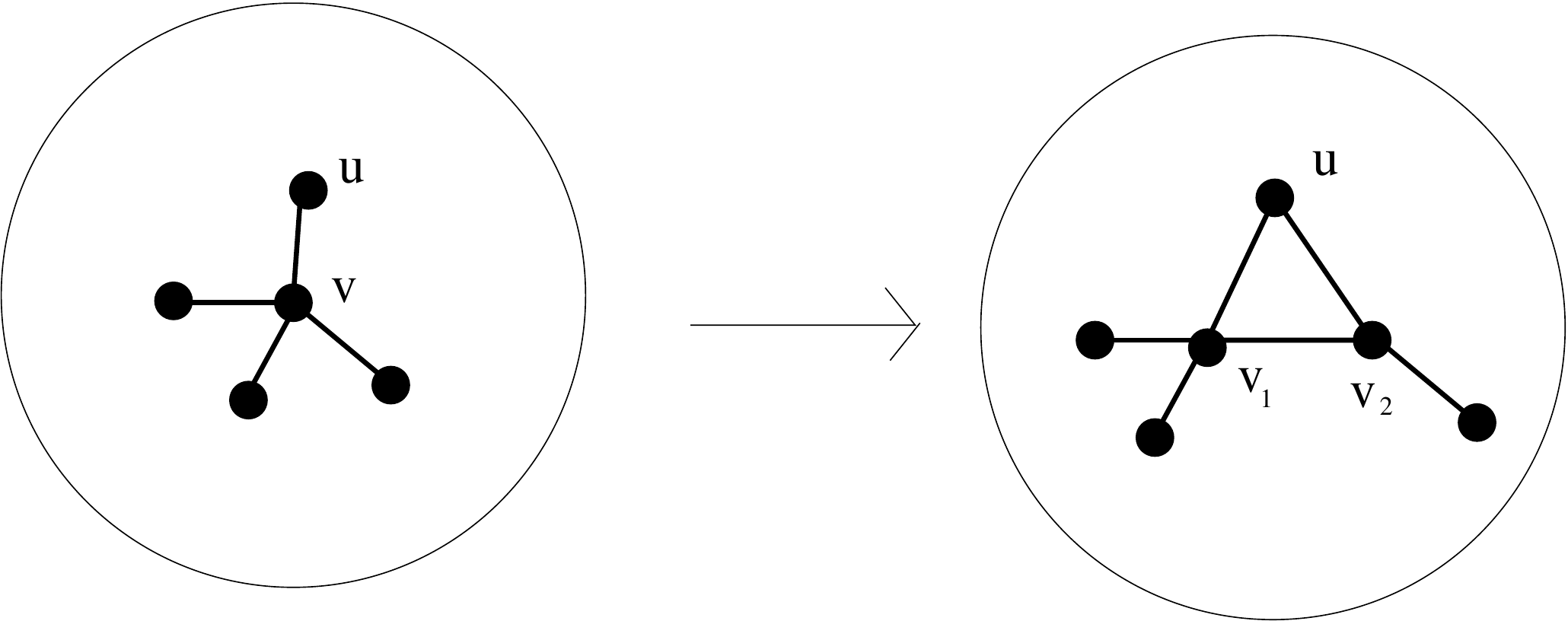}
\caption{The vertex splitting move.}
\label{Vertex splitting}
\end{figure}
\end{center}

The \emph{vertex splitting} operation removes an edge $uv$ and vertex $v$ and inserts a copy of $K_3$ on vertices $u,v_1,v_2$ and assigns all edges $xv$ into either an edge $xv_1$ or an edge $xv_2$. That is, let $G=(V,E)$, let $V^*:=V\setminus v$, let $E^*:=E(G[V^*])$ and let $E':=\{xy \in E: y=v\}$. Thus $G=(V,E)=(V^*+v,E^*+E'+uv)$ and if $G'$ is the result of a vertex split on the edge $uv$ then $G'=(V',E')=(V^*+\{v_1,v_2\},E^*+E_1'+E_2'+\{uv_1,uv_2,v_1v_2\})$ where $E_i':=\{xy \in E': x=v_i\}$ is an arbitrary partition of $E'$.
We refer to the inverse move as an edge contraction.

We remark that the vertex-to-$K_4$ move is a special case of the more elaborate \emph{graph extension move} employed in [14]. We also remark that the vertex splitting move is not needed for the proof of our main theorem.

Finally, an \emph{edge-joining} move \cite{N&O} combines two graphs $G$, $H$ to form a new graph with the vertices and
edges of these graphs together with an additional connecting edge $e=gh$ with $g$ in $G$ and $h$ in $H$.

%The proof of the inductive characterisation of $(2,1)$-tight graphs given in Theorem \ref{21}
%is proven along the following lines. By simple counting, if there are no inverse Henneberg $1$ moves then there are at least two vertices of degree $3$. If there are also no inverse Henneberg $2$ moves then it is shown that all such vertices are in subgraphs which are copies of $K_4$. In this case there is an admissible $K_4$-contraction move unless there are triangle obstructions  in the sense of there being inclusions $K_4\to K_4\sqcup K_3$. In this case contraction of the $K_4$ graph violates simplicity.  However, if triangle obstructions persist this leads to the existence of an admissible triangle contraction, that is, to an inverse vertex splitting move. In this way one can arrive at the following key lemma
%of \cite{N&O}) which in turn leads to Theorem \ref{21}.
%
%\begin{lem}
%Let $G$ be a $(2, 1)$-tight graph  which contains
%a copy of $K_3$. Then either $G = K_4$, $G$ has an admissible inverse vertex-splitting move, an admissible $K_4$ contraction, or every copy of $K_3$ is in a copy
%of $K_4 \sqcup K_4$ or $K_5 \backslash e$.
%\end{lem}

The \emph{vertex-to-$4$-cycle} move is a certain vertex splitting operation, as in Figure 4. The vertex $v_1$ is split to two vertices $v_1$ and $v_0$ and the edges $v_1v_2$ and $v_1v_3$ are duplicated as $v_0v_2$ and $v_0v_3$.
Other edges of the form $vv_1$ are either left or are replaced by $vv_0$.
The move preserves $(2,k)$-tightness and after the vertex split there will be no edge pair $wv_1, wv_0$ with $w \neq v_2,v_3$ and no edge $v_0v_1$.

\begin{center}
\begin{figure}[ht]
\centering
\includegraphics[width=8cm]{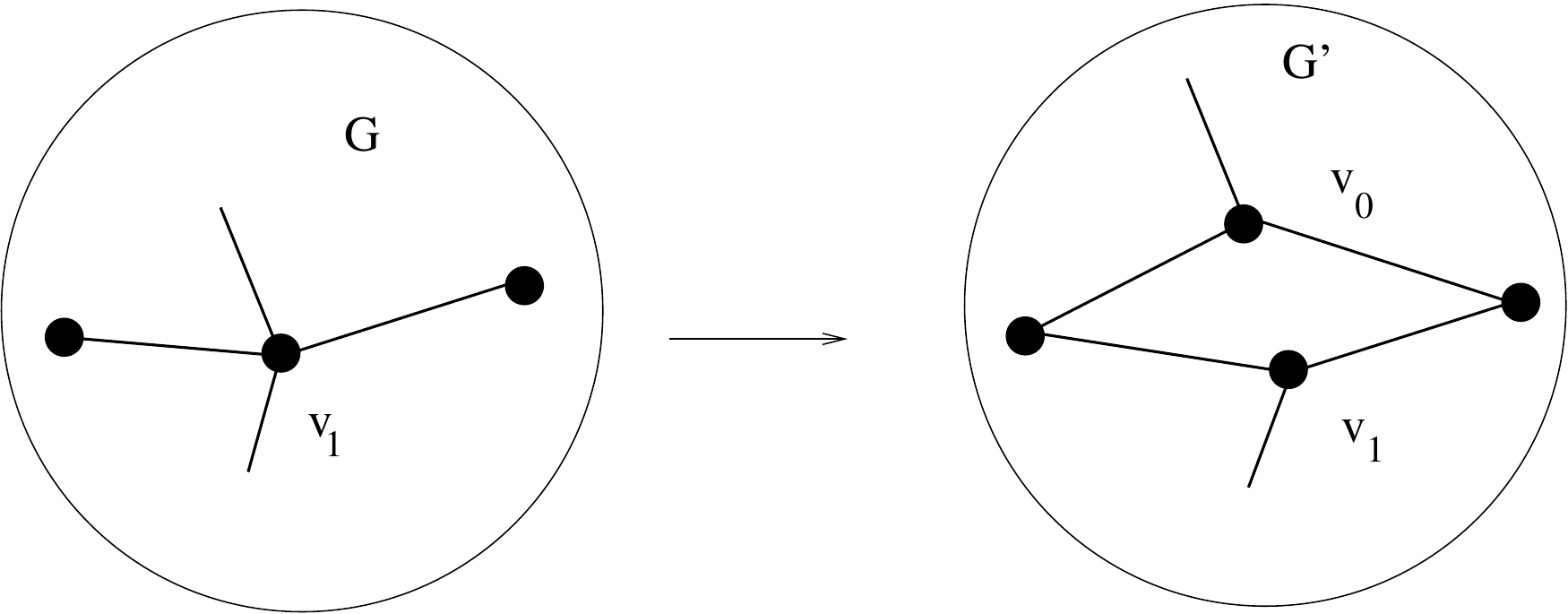}
\caption{The vertex-to-$4$-cycle move.}
\label{vertex-to-4-cycle}
\end{figure}
\end{center}

The inverse of such a move on a simple $(2,1)$-tight graph will be referred to as a $4$-cycle contraction.

For each of our moves we use the term \emph{admissible}, if whenever $G$ is simple and $(2,k)$-tight, the result of the move is also simple and $(2,k)$-tight. It is evident that inverse Henneberg 1 moves and, when $k=1$, inverse edge joining moves are always admissible. For each of the other moves it is easy to observe examples when the inverse moves are non-admissible.

We now prove Theorem \ref{t:21characterisation}.

%We now prove Theorem \ref{t:21characterisation} which gives a somewhat more efficient characterisation
%of simple $(2,1)$-tight graphs than Theorem \ref{21}.

\begin{proof}[Proof of Theorem \ref{t:21characterisation}] 
A basic incidence degree counting argument shows that if $G$ is $(2,1)$-tight with no possible inverse Henneberg $1$ moves then there are at least two degree $3$ vertices.  
%Indeed, if there are $n_k$ vertices of degree $k$ then we have $3n_3 + \dots +4n_4+\dots + rn_r= 2|E|=2(2|V|-1)$
%and so
%\[
%\sum_{k=1}^r (k-4)n_k =-2.
%\]
In fact these vertices must be in copies of $K_4$ graphs. If not then there is a potential
inverse Henneberg $2$ move and a counting argument shows that these
inverse moves are admissible. (This counting argument is given explicitly in \cite[Lemma $3.1$]{N&O}.)

Suppose now that one of these $K_4$ graphs does not have an admissible contraction to a vertex.
Then there is a vertex $w$ outside this $K_4$ with two edges to distinct vertices $a, b$ in this $K_4$, neither of which is of degree $3$ in $G$. Thus there is a $4$-cycle including $w, a, b$ and a degree $3$ vertex $v$ of the $K_4$ graph. The edge $vw$ is absent from $G$  and so there is a potential inverse $4$-cycle contraction, with $v \to w$, as long as the edge $wc$ is absent, where $c$ is the fourth $K_4$ vertex, since the move creates this edge.

Suppose first that $wc$ is absent. We claim that
the $4$-cycle contraction is admissible.

To check this we need only consider the count change for subgraphs $Y$ of $G$, with $f(Y)=2|V(Y)|-|E(Y)|=1$, which include $w$ and $c$
(and so receive $wc$) but not $v$ (since otherwise the count change is zero). There are three cases:

(i) $a$ and $b$ are in $Y$. However, adding $v$ and $3$ edges would create $Y^+$ with $f(Y^+)=0$, so $Y$ cannot have $f(Y)=1$.

(ii) $a$ and $b$ are not in $Y$. However, adding $a$ and $b$ and their $5$ edges $aw,ac,bc,bw,ab$ to $Y$ creates $Y^+$ with $f(Y^+)=0$.

(iii) Just $b$ (or $a$) is in $Y$. However, adding $a$ and the $3$ edges $ac,aw,ab$ creates $Y^+$ with $f(Y^+)=0$.

Thus this $4$-cycle contraction is admissible.

Suppose now that the edge $wc$ is present. Then $G$ contains a copy of $K_5\backslash e$ supported by the vertices $v,w,a,b,c$. Moreover assume that this is true for every degree 3 vertex, that is, there is no admissible inverse Henneberg 2, admissible $K_4$-to-vertex or admissible $4$-cycle contraction.
We show that if $G$ is not isomorphic to $K_5\backslash e$ then there is an inverse edge-joining move. This completes the proof of the theorem.

We argue as  in Lemma 4.10 of \cite{N&O}. Let $Y=\{Y_1,\dots ,Y_n\}$ be the subgraphs which are copies of $K_5\backslash e$. They are necessarily vertex disjoint since $f(Y_i\cup Y_j)=2-f(Y_i\cap Y_j)$ and every proper subgraph $X$ of $K_5\backslash e$ has $f(X) \geq 2$.
Let $V_0$ and $E_0$ be the sets of edges in $G$ which are in none of the $Y_i$. Then
\[
f(G)=\sum_{i=1}^nf(Y_i)+2|V_0|-|E_0|
\]
so $|E_0|=2|V_0|+n-1$. Since every vertex of degree less than 4 is contained in some $Y_i$, each vertex in $V_0$ is incident to at least $4$ edges.
If every $Y_i$ is incident to at least $2$ edges in $E_0$ then there are at least $4|V_0|+2n$ edge/vertex incidences in $E_0$. This
implies $|E_0|\geq 2|V_0|+n$, a contradiction. Thus
either there is a copy $Y_i$ with no incidences, which would imply $G=Y_i$, since $G$ is connected, contrary to our assumption, or there is a copy with one incidence. In this  case there is an inverse edge joining move, as desired.
\end{proof}

As mentioned in the introduction, a simpler version of the same proof scheme gives a construction for $(2,2)$-tight graphs as follows.

\begin{thm}\label{t:22characterisation} 
A simple graph $G$ is $(2,2)$-tight if and only if it can be obtained from $K_1$ by the sequential application of
moves of $4$ types, namely the Henneberg $1$ and $2$ moves, the vertex-to-$K_4$ move and the vertex-to-$4$-cycle move.
\end{thm}

\begin{proof}
Suppose that $G$ is simple and $(2,2)$-tight with at least one edge.
Then, by simple counting, there exist vertices of degree $2$ or $3$. As in the last proof, if there are no inverse Henneberg moves then it
follows that the minimum degree is $3$ and any such vertex $v$ must lie in a copy of $K_4$. As above if there is no admissible contraction on this $K_4$ then there is a vertex $w$ adjacent to two vertices $a,b$ in that copy of $K_4$. Since $G$ is $(2,2)$-sparse there is no edge $wc$, with $c$ the final $K_4$ vertex, and hence, the argument above, shows there is an admissible $4$-cycle contraction move.
\end{proof}

\section{Henneberg moves on frameworks on surfaces}
\label{s:hennebergmoves}

The role played by the Henneberg moves in rigidity theory is extensive and well studied  (see for example \cite{B&J,Hen,J&J,Lam,Tay2,whi-book})
and we now consider such moves on frameworks on surfaces.
The case of the Henneberg $1$ move is elementary. Recall that an isostatic framework is one which is independent and infinitesimally rigid.

\begin{lem}\label{Hen1}
Let $(G,p)$ be a generic framework on an algebraic surface $\M$, let $G'$ be a graph
obtained from $G$  by  a Henneberg $1$ move with new vertex $v$, and let $p'=(p,p_v)$ be  generic. Then $(G,p)$ is isostatic on $\M$ if and only if $(G',p')$ is isostatic on $\M$.
\end{lem}

\begin{proof} The rigidity matrix
$R_\M(G',p')$ contains $3$ new rows and $3$ new columns and those columns are zero everywhere
in the $|E|+|V|$ rows for $G$.
Reorder the rows and columns so that the first three rows and columns are the new ones.
By the generic location of $p_v$ and the block upper triangular structure the first three
rows are independent of the rest. Thus
there is a row dependency in $R_\M(G',p')$ if and only if there is a row dependency in $R_\M(G,p)$.
\end{proof}

The preservation of independence and isostaticity under the Henneberg $2$ move is considerably more subtle. To see an aspect of this
let $(G,p)$ be generic on $\M$, where $\M$ is the cylinder surface $x^2+y^2=1$ in $\bR^3$. Examining
the form of $R_\M(G,p)$, it is evident that the addition of a degree $0$ vertex
increases the rank by $1$.
Elementary linear algebra also shows that the
addition of a degree $1$ framework point $(x_1,y_1,z_1)$ incident to a point $(x,y,z)$ increases the rank by $2$ if and only if
$(x_1,y_1,z_1)$ is not equal to $(x,y,z)$ or $(-x,-y,z)$.
More surprisingly, similar considerations show that there are only four possible points where the rank does not fully increase for a new degree $2$ vertex. This is contrary to the situation in the plane where any point on a line through the existing edge will create a copy of $K_3$ whose rows give a minimal linear dependency.
In view of this, adapting a typical Henneberg $2$ argument for rigidity preservation would require placing the new degree $3$ vertex at one of a finite number of points.
However, the dependencies created by each of these points are not amenable to such an argument.

Such difficulties  motivated us in \cite[Section $4$]{NOP}, to show that continuous rigidity was preserved.
The proof there makes particular use of the two trivial motions of the cylinder and it is not clear how one might generalise this method
to surfaces with fewer trivial
motions. Our first new approach below is based on the convergence of specialised frameworks. A direct algebraic proof is given in Section \ref{sec:algh2}.

The following notation and observations will be useful.

Let $\T_{q_i}$ denote the tangent space for the point $q_i$
on the surface $\M$ and for a framework vector $q=(q_1,q_2,\dots ,q_n)$ let
\[
\T_q =\T_{q_1}\oplus \T_{q_2} \oplus \dots \oplus \T_{q_n}\subseteq \bR^3 \oplus \dots \oplus \bR^3 = \bR^{3n}
 \]
denote the joint tangent space. This is the space of infinitesimal velocities
of any framework with framework vector $q$.
With the usual Euclidean structure we have
orthogonal projections
\[
P_q: \bR^{3n} \to \T_q,
\mbox{   }
F_q: \bR^{3n} \to \F(G,q), \mbox{   }
 Q_q: \bR^{3n} \to \R_q,
 \]
where $\F(G,q)$ is the vector space of infinitesimal flexes of the framework $(G,q)$ and where $\R_q$ is the subspace of $\T_q$ consisting of rigid motion infinitesimal flexes.
We have $\R_q = \ker R_\M(K_n, q)$ for $n\geq 6-k$, where $k$ is the type of $\M$.
%sufficiently large $n$. Note $n\geq 4$ suffices for a surface of type 2 and $n\geq 5$ suffices for a surface of type 1.
Since $\M$ is smooth the function $q \to P_q$ is continuous in a neighbourhood of $p$. This also holds  in the case of a \textit{degenerate framework vector} $p$, in the sense that some or all of the vectors $p_i$ may agree.
Similarly, if the spaces $\R_q$ have dimension $k$ throughout a neighbourhood
of a (possibly degenerate) framework vector $p$ then the function $q \to Q_q$ is continuous, as long as the degeneracy includes three non-collinear points.
In general the function $q \to \F(G,q)$ is lower semi-continuous.

\begin{lem}\label{hen2independent}
Let $G$ be a simple graph, let $G'$ be derived from $G$ by a Henneberg 2  move and let $\M$ be an irreducible surface. If $G$  is minimally infinitesimally rigid  on $\M$ then  $G'$ is  minimally infinitesimally rigid  on $\M$.
\end{lem}

\begin{proof}Let $(G,p)$ be generic on $\M$ with $p=(p_1,\dots,p_n)$ and let $p'=(p_0,p)$
where $(G',p')$ is generic on $\M$. We let  $v_1v_2$ denote the edge involved in the Henneberg move and write $v_0$ for the new vertex. 
Suppose that $(G',p')$ is not infinitesimally rigid on $\M$. Then it follows that every specialised framework on $\M$ with graph $G'$ is infinitesimally flexible. Figure
\ref{Henn2new} indicates a sequence of specialisations $(G', p^k)$ in which only the $p_0$ framework
point is specialised to a point $p_0^k$. Also  $p_0^k$ tends to $p_2$ 
in the direction $a$ where $a$ is a tangent vector at $p_2$ which is orthogonal to a tangent vector $b$ at $p_2$ where $b$ is orthogonal to $p_2-p_1$. 
More precisely, the normalised vector $(p_2^k-p_0)/\|p_2^k-p_0\|$ converges to $a$, as $k\to \infty$. Each of the frameworks  $(G', p^k)$ has a unit norm flex $u^k$ which is orthogonal to the space of rigid motion infinitesimal flexes of its framework. By the Bolzano-Weierstrass theorem there is a subsequence of the sequence $u^k$ which converges to a vector, $u^\infty$ say, of unit norm. Discarding framework points and relabelling we may assume this holds for the original sequence. The limit flex of the degenerate framework $(G',p^\infty)$ has the form
\[
u^\infty=(u_0^\infty,u_1,u_2,\dots ,u_n),
\]
while the degenerate framework vector is
\[
p^\infty =(p^\infty_0,p_1,p_2,p_3,\dots ,p_n),
\]
where $p^\infty_0 = p_2$.

We claim that the velocities $u_1, u_2$ give an infinitesimal flex of $p_1p_2$
(as a single edge framework on $\M)$. To see this note that in view of the
well-behaved convergence of $p_0^k$ to $p_2$ (in the $a$ direction) it follows that the velocities $u_2$ and $u_0^\infty$ have the same component in the $a$ direction, and so $(u_2 -u_0^\infty).a=0$. Since $u_2 -u_0^\infty$ is tangential to $\M$ it follows from the choice of $a$ that $u_2 -u_0^\infty$ is orthogonal to $p_2-p_1$. On the other hand
$u_1-u^\infty_0$ is orthogonal to $p_2-p_1$ and so taking differences
$u_2 -u_1$  is orthogonal to $p_2-p_1$ as desired.

It now follows, by the rigidity of $(G,p)$ that the velocity vector
\[
u^\infty_{\rm res}=(u_1,u_2,\dots ,u_n),
\]  
is a rigid motion flex. However $u^\infty_0$ is determined by $u_1$ and $u_3$, since the bars for the pairs $p_0^\infty, p_1$ and $p_0^\infty, p_3$ are present. Thus $u^\infty_0$  agrees with $u_2$ and  $u^\infty$  is a rigid motion flex for $(G', p^\infty)$.
This is a contradiction since the flex has unit norm and is orthogonal to the rigid motion flexes.
\end{proof}

\begin{center}
\begin{figure}[ht]
\centering
\includegraphics[width=8cm]{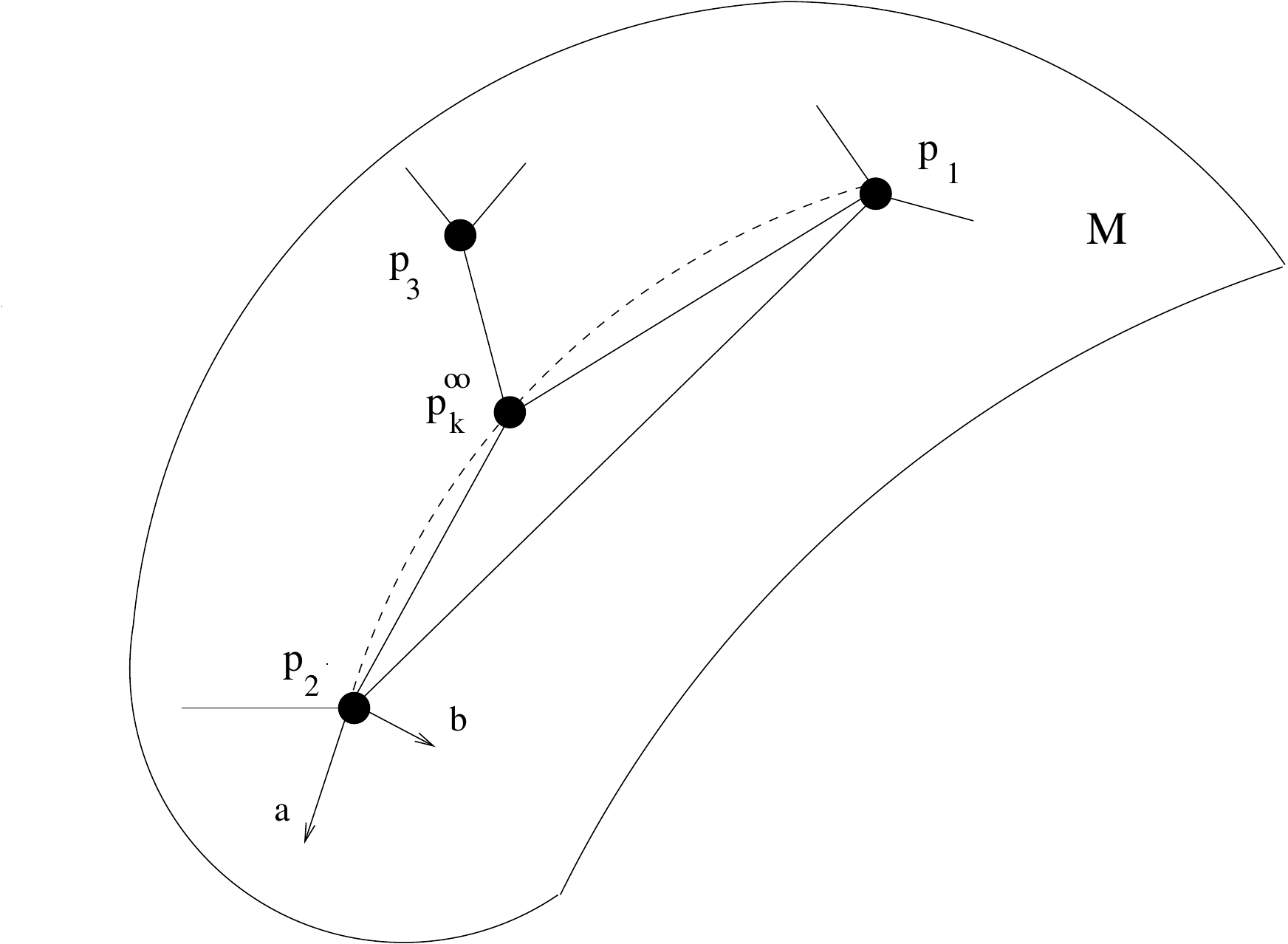}
\caption{Geometry for the sequential specialisation proof of rigidity preservation under the Henneberg $2$ move.}
\label{Henn2new}
\end{figure}
\end{center}

\section{Vertex surgery moves on frameworks}
\label{vertexmoves}

Let $\M$ be an irreducible surface of type $2,1$ or $0$. We now show the preservation of independence of $G$ on $\M$ under the vertex splitting move, the vertex-to-$K_4$ move and the vertex-to-$4$-cycle move.
The form of argument in the proofs is in terms of
flexibility preservation for the inverse move and this is obtained once again by the consideration
of certain well-behaved contraction sequences $(G', p^k) \to (G', p^\infty)$.

For completeness and comparison we include Lemma 5.1 although it will not be needed for the proof of the main theorem.

\subsection{Vertex Splitting}
\label{vertexsplitsec}
Recall that in a vertex-splitting move $G\to G'$
a vertex $v$ is doubled to $v_1$ and $v_2$ with a doubling of an edge $vu$ to
$v_1u, v_2u$, with the remaining edges to $v$ distributed arbitrarily
between $v_1$ and $v_2$.
The reverse move $G'\to G$ is an admissible triangle contraction
in the sense that the result of collapsing an edge (of a triangle) to coincident endpoints
results in a simple $(2,k)$-tight graph.

The preservation of independence under vertex splitting for bar-joint frameworks in various dimensions
was shown by Whiteley  \cite{whi-vertexsplit}. The arguments there were  based on self-stresses and $k$-frames. Here we argue somewhat more directly in terms of infinitesimal flexes.

\begin{lem}\label{Vertexsplitprop}
Let $\M$ be an irreducible surface of type $2, 1$ or $0$
and let $G \to G'$ be a vertex splitting move. If $G$ is independent on $\M$ then
$G'$ is independent on $\M$.
\end{lem}

\begin{proof}
To fix notation, let $G$ have $n-1$ vertices $v, v_3, v_4,\dots ,v_n$
with vertex $v$ to be split and $v_3$ featuring as the vertex $u$ in the vertex-splitting move.
Let $p,p'$ be generic vectors for $G, G'$ respectively, with $p=(p(v),p(v_3),\dots ,p(v_n))= (p_1,p_3,\dots ,p_n)$ and
$p'=(p(v_1)',p(v_2)', \dots , p(v_n)')=(p_1,p_2,\dots ,p_n)$.

It will be sufficient to show that if  $G'$ is dependent on $\M$ then so too is $G$. Accordingly we assume that $(G',p')$ is a generic framework which is infinitesimally flexible.
Let $p_2^k$ be a sequence of generic points on $\M$ which converges to $p_1$ in the following well-behaved manner, namely the unit vectors $a_k$ in the directions $p_2^k-p_1$ converge to a unit vector $a$ with $a. (p_3-p_1)\neq 0$.
Also, let $p^k=(p_1, p_2^k, p_3, \dots ,p_n)$ and let $p^\infty = (p_1, p_1, p_3, \dots , p_n)$. By the assumption, for each $k$ there exists a unit vector
$
u^k=(u_1^k, \dots ,u_n^k)
$
in $\T_{p^k}$ which is an infinitesimal flex of $(G', p^k)$ and which has no rigid motion flex component, in the sense that $Q_{p^k}u^k=0$. Furthermore, taking subsequences, we may assume that $u^k$ converges to some unit norm flex $u'=(u_1,\dots ,u_n)$ as $k$ tends to infinity. By the flex condition we have
\[
u_1.(p_1-p_3)=u_3.(p_1-p_3)
\]
and
\[
u_2.(p_1-p_3)=\lim_{k\to \infty}u_2^k.(p_2^k-p_3)=\lim_{k\to \infty}u_3^k.(p_2^k-p_3)=u_3.(p_1-p_3).
\]
On the other hand, since $a_k$ is a scalar multiple of $p_2^k-p_1$,
\[
(u_2-u_1). a= \lim_{k\to \infty}(u_2^k-u_1^k). a_k =
 \lim_{k\to \infty}(u_2^k-u_1^k). (p_2^k-p_1) =0.
\]
Thus $u_1=u_2$ since both vectors lie in the same tangent plane and their difference is orthogonal to
the linearly independent vectors $a$ and $p_1-p_2$.

Note that there is always a natural injective map
\[
\iota : \F(G,p) \to  \F(G',p^\infty),
\]
provided by the map $(w_1, w_3, w_4, \dots , w_n) \to (w_1, w_1, w_3, w_4, \dots ,w_n)$.
Also a flex of $(G',p^\infty)$ of the form $(w_1, w_1, w_3, w_4, \dots ,w_n)$  determines the flex $(w_1, w_3, w_4, \dots ,w_n)$ of $(G, p)$. Thus the flex $u'$ gives rise to 
a flex, $u$ say, of $(G, p)$. From the continuity of $q \to Q_q$ it follows  that $Q_{p^\infty}u'=0$. Thus $Q_pu=0$ and so $u$ is a proper flex and $G$ is dependent, as required.
\end{proof}

\subsection{Vertex-to-$K_4$ move.}
\label{vtok4secsub}
We now show preservation of independence under the vertex-to-$K_4$ move $G\to G'$. Once again we will argue in terms of a  \textit{well-behaved contraction} $(G', p^k) \to (G', p^\infty)$, this time  with respect to a particular sequence of  subframeworks $(K_4,(p_1,p_2^k,p_3^k ,p_4^k))$
with $p_i^k\to p_1$ as $k\to \infty$, for each $i=2,3,4$.

First we consider separately the case of a cylinder since here we may
exploit the two trivial infinitesimal motions to obtain a simple direct proof.

\bigskip

\begin{lem}\label{l:vK4lemma}
Let $\M$ be the cylinder, let $G'$ be  $(2,2)$-sparse with proper subgraph $K_4$ and
let $p=(p_1, \dots ,p_n)$ be a framework vector in $\M^n$ for $G'$, with $p_1,\dots , p_4$ the placement of the vertices for $K_4$.
Let
\[
p^k=(p_1,p_2^k, p_3^k, p_4^k, p_5, \dots ,p_n)
\]
be a sequence in $\M^n$ with $p_i^k \to p_1$ as $ k\to \infty$, for $i=2,3,4$,
and suppose that

(i) each framework $(K_4,(p_1,p_2^k,p_3^k ,p_4^k))$ has $2$ independent infinitesimal flexes on $\M$,

(ii) the dimension of the infinitesimal flex space of $(G',p^k)$ is greater than $2$ for each $k$.

Then there is a unit norm infinitesimal flex $u'$ of $(G',p^\infty)$ of the form
\[
u'=(0, 0, 0, 0, u_5,\dots ,u_n).
\]
\end{lem}

\begin{proof}
In view of (ii) for each $k$ there exists an
 infinitesimal flex
$
u^k=(u_1^k, \dots ,u_n^k)
$
of $(G,p^k)$ in $\bR^{3n}$ with norm $\|u^k\|=1$, and
with $u_1^k=0$ for all $k$.
Thus, from (i) it follows that
$u_i^k=0$ for $i=2,3,4$.
Taking subsequences we may assume that $u^k$ converges to a unit norm flex $u'$ of $(G',p^\infty)$ and this has the desired form.
\end{proof}

\begin{cor}\label{vertex2k4cylinder}
The vertex-to-$K_4$ move on the cylinder preserves generic infinitesimal rigidity.
\end{cor}

\begin{proof}
Let $G\to G'$ be a vertex-to-$K_4$ move. Suppose that $G'$ is dependent.
Let $p^k$ be a sequence of generic framework vectors for $G'$ as specified in Lemma \ref{l:vK4lemma}, with $p_i^k \to p_1$ for $i=1,\dots , 4$. Since $p^k$ is generic (i) holds, and (ii) holds since $G'$ is dependent.
Thus, by the lemma  $(G',p^\infty)$ has an infinitesimal flex of the indicated form. This in turn gives an infinitesimal flex $u$ of $(G,p)$ with
$u_1=0$ and so $G$ is dependent, as required.
\end{proof}

We now consider the much more subtle case of irreducible surfaces which are of type $1$ or $0$.
In this case we exploit the characteristic property that at a generic point $p_1$ on $\M$ the principal curvatures are distinct.

Let $(K_4,(p_1,\dots ,  p_4))$ be a bar-joint framework on the irreducible surface $\M$
and assume that the principal curvatures $\kappa_s, \kappa_t$ at $p_1$ are well-defined, with $\kappa_s \neq \kappa_t$, and that $\hat s$ and $ \hat t$ are associated orthonormal vectors in the tangent
plane at $p_1$. For definiteness let $\hat n$ be the unit normal at $p_1$
with $\hat s, \hat t,  \hat n$ a right-handed orthonormal triple. By Taylor's theorem,
in a neighbourhood of $p_1$ the points $p$ on $\M$ take the form
\begin{equation}\label{taylor}
p(s,t)=p_1+(s\hat s+t\hat t) +1/2(\kappa_s s^2+\kappa_t t^2) \hat n + r(s,t)
\end{equation}
with $\|r(s,t)\|=O(\|(s,t)\|^3)$, for $\|(s,t)\|\leq R$, say.
In particular
\[
\frac{dp}{ds}=\hat s +\kappa_ss \hat n +r_s,\mbox{  } \frac{dp}{dt}=\hat t +\kappa_tt \hat n +r_t
\]
where $\|r_s\|$ and $\|r_t\|$ are of order $\|(s,t)\|^2$.
Also the vectors
\begin{equation}\label{n(s,t)*}
n(s,t) = dp/ds(s, t) \times dp/dt(s,t)
\end{equation}
give a continuous choice of normal vectors in a neighbourhood of $p_1$, with
\[
n(s,t) =
\hat s \times \hat t + \kappa_tt\hat s \times  \hat n + \kappa_s s  \hat n\times \hat t + \ul{r}
\]
\begin{equation*}
=  \hat n -(\kappa_tt\hat t + \kappa_ss\hat s) + \ul{r}
\end{equation*}
where $\|\ul{r}\|$ is of order $\|(s, t)\|^2$ in this neighbourhood.

Assume now that the triple $p_2, p_3, p_4$ is generic and lies in this neighbourhood with
$p_i=p(s_i,t_i)$. For $\epsilon \leq 1$ let $p_i^\epsilon = p(\epsilon s_i, \epsilon t_i)$.
By a \textit{well-behaved $K_4$ contraction} of  $(G,p)$ over the  subgraph
$K_4 \subseteq G$, with vertices $v_1, \dots ,v_4$, we mean a framework sequence
$(G, p^k)$ with
\[
p^k=(p_1,p_2^{\epsilon_k},p_3^{\epsilon_k},p_4^{\epsilon_k},p_5,\dots ,p_n),
\]
where ${\epsilon_k} \to 0$ as $k\to \infty$ and where the local coordinates
$s_2, t_2, s_3, t_3, s_4, t_4$ satisfy the determinant condition
\[\left| \begin {array}{ccc}
s_2&t_2&s_2t_2\\
s_3&t_3&s_3t_3\\
s_4&t_4&s_4t_4
\end {array} \right| \neq 0.
\]
It is straightforward to see that we can choose a well-behaved $K_4$ contraction. For example if $s_i=i,t_i =i^2$ for $i=2,3,4$ then the determinant has the value 48.

\begin{lem}\label{l:justK4}
Let $\M$ be an irreducible surface of type 1 or 0.
Let $(K_4,(p_1,p_2^k, p_3^k, p_4^k)), k=1,2,\dots $, be a well-behaved contraction of frameworks
on $\M$ and let $u_k,  k=1,2,\dots $, be an associated sequence of infinitesimal flexes which forms a convergent sequence in $\bR^{12}$. Then the limit vector has the form $(u_1,u_1,u_1,u_1)$.
\end{lem}

\begin{proof} Let $u=(u_1,\dots , u_4)$ be an infinitesimal flex of $(K_4, p)$. Equivalently,
$u_i.n_i=0$ where
$n_i$ is the unit normal at $p_i$ and
$(p_i-p_j).(u_i-u_j)=0$  for $1 \leq i<j \leq 4$.
Since $(K_4, p)$ is infinitesimally rigid  in $\bR^3$ the flex
$u$ is equal to $u_a+u_b$ where $u_b$ is determined by translation by the vector $b$ and where
$u_a$ corresponds to an infinitesimal rotation about a line through $p_1$ with direction
vector $a$. Thus $u_1= b$ and we may choose the magnitude and direction of $a$ so that
$u_i-u_1= (p_i-p_1)\times a $, for $i=2,3,4$.
Substituting gives
$$(a \times (p_i-p_1)).n_i = u_1.n_i,$$
or equivalently,
$$ a.(n_i \times (p_i - p_1))+b.n_i=0,$$ for $i=2,3,4$.

We have the  normal vectors
$n(s,t) = dp/ds(s_i, t_i) \times dp/dt(s_i,t_i)$ as in Equation \ref{n(s,t)*} above.
At the point $p_i^\epsilon = p(\epsilon s_i, \epsilon t_i)$ these normals take the form
\[
n^\epsilon_i = n(\epsilon s_i, \epsilon t_i)= \hat n -\epsilon(\kappa_tt_i\hat t + \kappa_ss_i\hat s) + \ul{r}_i^\epsilon
\]
where $\|\ul{r}_i^\epsilon \| = O(\epsilon^2).$

Consider now an infinitesimal flex $u^\epsilon$ of the framework $(K_4, p^\epsilon)$ on $\M$. The associated equations are
\begin{eqnarray}\label{eqn:3}
a^\epsilon.(n_i^\epsilon \times (p_i^\epsilon-p_1))+b^\epsilon.n_i^\epsilon=0,
\end{eqnarray}
for $i=2,3,4$,
and we may identify the crossed product here as
\[
 n_i^\epsilon \times (p_i^\epsilon-p_1) =  (\hat n -\epsilon(\kappa_tt_i\hat t + \kappa_ss_i\hat s))\times(\epsilon(s_i\hat s+t_i\hat t) +1/2\epsilon^2(\kappa_s s_i^2+\kappa_t t_i^2) \hat n) + R_i^\epsilon
\]
with $\|R_i^\epsilon\| = O(\epsilon^3)$.

We may assume, by passing to a subsequence, that $\epsilon$ runs through a sequence $\epsilon_k$ tending to zero
and that the associated unit norm flexes $u^\epsilon$ converge to a limit flex
$u^0$  of the
degenerate framework $(K_4, (p_1,p_1,p_1,p_1))$ on $\M$. Let $b^0=u^0_1$ and let $b^\epsilon $ and $a^\epsilon$ be the associated vectors. While $b^\epsilon = u^\epsilon_1$ converges to $b^0$, as $\epsilon = \epsilon_k \to 0$,  the sequence $(a^{\epsilon_k})$ may be unbounded. However,
in view of the three equations
\[
u_i^\epsilon - u_1^\epsilon=(p_i^\epsilon - p_1) \times a^\epsilon
\]
and the definition of $p_i^\epsilon$
it follows that $\|a^{\epsilon_k}\|$ is at  worst of order $1/\epsilon_k$. We shall show that $\|a^{\epsilon_k}\|$ is in fact bounded and so, from the equation above, the desired
conclusion follows.

Returning to the three equations, see Equation (\ref{eqn:3}), which determine $a^\epsilon$ from $b^\epsilon$ we have
\[
a^\epsilon.(\epsilon s_i \hat{t}-\epsilon t_i\hat s-\epsilon^2 s_i t_i(\kappa_s-\kappa_t) \hat n +R_i^\epsilon)
-\kappa_s(b^\epsilon.\hat s)\epsilon s_i-\kappa_t(b^\epsilon.\hat t)\epsilon t_i+r_i^\epsilon=0,
\]
where $r^\epsilon_i=\|b^\epsilon . r_i^\epsilon\| = O(\epsilon^2)$.
Note that $\|a^\epsilon.R^\epsilon_i \|=O(\epsilon^2)$ and so it follows,
introducing coordinates for $a^\epsilon$, and cancelling a factor of $\epsilon$, that
\[
(a_s^\epsilon\hat s+a_t^\epsilon\hat t +a_n^\epsilon \hat n).(s_i\hat t-t_i\hat s-\epsilon s_it_i(\kappa_s-\kappa_t) \hat n)
-\kappa_s(b^\epsilon.\hat s)s_i-\kappa_t(b^\epsilon.\hat t)t_i=O(\epsilon),
\]
for $i=2,3,4$.
Thus
\[
-a^\epsilon_st_i + a^\epsilon_ts_i -a_n^\epsilon \epsilon s_it_i(\kappa_s-\kappa_t)
= d^\epsilon_i,\quad \mbox{for } i=2,3,4,
\]
where
\[
 d^\epsilon_i =b^\epsilon.(\kappa_ss_i\hat s + \kappa_tt_i\hat t) + X_i^\epsilon,
\]
with $ X_i^\epsilon =O(\epsilon).$

Let $\eta= \epsilon (\kappa_s-\kappa_t)$ for $i=2,3,4$, let $A_\epsilon$ be the matrix
\[\left[ \begin {array}{ccc}
-t_2&s_2&-s_2t_2\eta\\
-t_3&s_3&-s_3t_3\eta\\
-t_4&s_4&-s_4t_4\eta
\end {array} \right],
\]
and note that $\det A_\epsilon = C\epsilon$ for some nonzero constant $C$. By Cramer's rule
we have
\[
a_n^\epsilon = (\det A_\epsilon)^{-1}\left| \begin {array}{ccc}
-t_2&s_2&d_2^\epsilon\\
-t_3&s_3&d_3^\epsilon\\
-t_4&s_4&d_4^\epsilon
\end {array} \right|
= (\det A_\epsilon)^{-1}\left| \begin {array}{ccc}
-t_2&s_2&X_2^\epsilon\\
-t_3&s_3&X_3^\epsilon\\
-t_4&s_4&X_4^\epsilon
\end {array} \right|,
\]
since the column for $d_i^\epsilon - X_i^\epsilon$ is a linear combination of the first two columns. It follows that the sequence $a_n^{\epsilon_k}$ is bounded.

The boundedness of $(a_s^{\epsilon_k})$, and similarly $(a_t^{\epsilon_k})$, follows more readily, since
\[
a_s^\epsilon = (\det A_\epsilon)^{-1}\left| \begin {array}{ccc}
d_2^\epsilon&s_2&-s_2t_2\eta\\
d_3^\epsilon&s_3&-s_3t_3\eta\\
d_4^\epsilon&s_4&-s_4t_4\eta
\end {array} \right|
\]
and the $\epsilon$ factors cancel. Thus, the sequence of vectors $a^{\epsilon_k}$ is bounded, as desired.
\end{proof}

\begin{lem}\label{p:k4onM}
Let $\M$ be an irreducible surface  of type $1$ or $0$, let $G'$ be $(2,1)$-sparse with $v_1,\dots ,v_4$ inducing a $K_4$ subgraph, and
let $p=(p_1, \dots ,p_n)$ be a generic framework vector in $\M^n$.
Let
\[
p^k=(p_1,p_2^k, p_3^k, p_4^k, p_5, \dots ,p_n)
\]
be sequence in $\M^n$, with $p_i^k \to p_1$  as $ k\to \infty$,  for $i=2,3,4$,
such that $(G, p^k)$ is a well-behaved contraction with limit  $(G', p^\infty)$.
If the rigid infinitesimal motion spaces $\R_{p^k}$ and $\R_{p^\infty}$ are one-dimensional
and the dimension of the flex space $\F(G',p^k)$ is greater than $1$ for all $k$,
then there is a unit norm flex $u$ in  $\F(G',p^\infty)$ which is orthogonal to $\R_{p^\infty}$
and satisfies $u_1=u_2=u_3=u_4$.
\end{lem}

\begin{proof}
By the hypotheses for
each $k$ there exists an
infinitesimal flex
$
u^k=(u_1^k, \dots ,u_n^k)
$
of $(G,p^k)$ lying in the multiple tangent space $\T_{p^k}$
such that the Euclidean norm  of $u^k$ is unity and
$u^k$ is orthogonal to the subspace $\R_{p^k}$.
Taking a subsequence if necessary we may assume that $u^k$ converges to $u$ as $k \to \infty$.
By Lemma \ref{l:justK4} the velocities $u_1,\dots ,u_4$ agree.
By the hypotheses, the orthogonal projections $Q_k$ onto  $\R_{p^k}$
converge to the projection $Q_\infty$ onto  $\R_{p^\infty}$ and so $u$ is orthogonal
to $\R_{p^\infty}$, as desired.
\end{proof}

\begin{cor}\label{vertex2k4}
The vertex-to-$K_4$ move for an irreducible surface of type $1$ or $0$ preserves generic infinitesimal rigidity.
\end{cor}

\begin{proof}
This follows from the previous lemma in the same manner as the proof of 
Corollary \ref{vertex2k4cylinder}.
\end{proof}

\subsection{The vertex-to-$4$-cycle move.}

\begin{lem}\label{l:Vertexto4cycle}
Let $\M$ be an irreducible surface of type $k$
and let $G \to G'$ be a vertex-to-$4$-cycle move. If $G$ is minimally infinitesimally rigid  on $\M$ then
$G'$ is  minimally infinitesimally rigid on $\M$.
\end{lem}

\begin{proof}
Once again we use a sequential contraction argument. Let $G$ have $n$ vertices $v_1, v_2, ,\dots ,v_n$
and edges $v_1v_2, v_1v_3$ and let $G\to G'$ be the move in question, with new vertex $v_0$ and edges $v_0v_2, v_0v_3$. It will be sufficient to show that if  $G'$ is dependent on $\M$ then so too is $G$.

Let $p, p'$ be the generic framework vectors for $G, G'$ respectively, with $p'=(p_0,p_1,\dots ,p_n)$.
Also let $p^k=(p_0^k,p_1, \dots ,p_n)$ be generic, with $p_0^k$ converging $p_1$.
% UNNEC tangentially, in the sense that the unit vector in the direction $p_0^k-p_1$ converges to a fixed %vector $a$ which is tangential to $\M$ at $p_1$ and is such that the three vectors $a, p_2-p_1,p_3-p_1$ are %linearly independent. (In case $\M$ is a plane we just require that $a$ is not dependent on either of the %other vectors.)
By the assumption  for each $k$ there exists a unit vector
$
u^k=(u_0^k, u_1^k, \dots ,u_n^k)
$
in the joint tangent space $\T_{p^k}$ which is an infinitesimal flex of $(G', p^k)$ and which is orthogonal to the rigid motion flexes. In earlier notation, $Q_{p^k}u^k=0$.  Taking subsequences, we may assume that $u^k$ converges to some unit norm flex $u'=(u_0, u_1,\dots ,u_n)$ of the degenerate framework $(G',p^\infty)$, as $k$ tends to infinity, where $p^\infty = (p_1,p_1,p_2,\dots ,p_n)$. Also, by the assumption on $G$, this degenerate framework (for $G'$) has a space of rigid motion flexes which is naturally identifiable with the space of rigid motion flexes of  $(G,p)$. It remains to show that $u_0=u_1$ so that we may conclude that $(u_1,\dots ,u_n)$ is a proper flex of $(G,p)$, completing the proof.

It follows from the flex conditions and taking limits that
$u_0-u_2$ is orthogonal to $p_1-p_2$, and  $u_0-u_3$ is orthogonal to $p_1-p_3$.
Also $u_1-u_2$ is orthogonal to $p_1-p_2$, and  $u_1-u_3$ is orthogonal to $p_1-p_3$.
It follows, subtracting, that $u_0-u_1$ is orthogonal to $p_1-p_2$ and to $p_2-p_3$.
At the same time $u_0-u_1$ lies in the tangent plane at $p_1$ and we may choose $p_2, p_3$ so that
$0$ is the only tangent vector orthogonal to  $p_1-p_2$ and to $p_2-p_3$.

\end{proof}

\section{The algebraic approach}
\label{sec:algh2}

We now give a direct algebraic proof of the preservation of infinitesimal rigidity under the Henneberg $2$ move on an irreducible surface.
We expect this approach to be more widely useful in the analysis of 
bar-joint frameworks in higher dimensions.
\medskip

Assume that $\M$ is an irreducible surface of type $k$ which is defined by the irreducible polynomial $m(x,y,z)=0$ where the
coefficients of $m$ are in $\bQ$.
Suppose that $G$ is a $(2,k)$-tight graph and that $(G,p)$ is a generic framework on $\M$ with $p=(p_1,\dots,p_n)$. Also let $p^+=(p,p_v)$
where $(G^+,p^+)$ is generic on $\M$ and $G^+$ derives from $G$ through a Henneberg $2$ move. We write  $v_1v_2$ for the edge involved in the Henneberg move and $v$ for the new vertex. 

Since $G$ is independent the rigidity matrix
$R_\M(G\backslash v_1v_2,p)$ has a flex vector $u=(u_1,\dots,\break u_n)$ in the nullspace which is not a flex of
$(G,p)$. In particular $(p_1-p_2).(u_1-u_2) \neq 0$.
Moreover we may choose $u$ as a solution of
the equations $R_\M(G,p)u=A$ where $A$ is a column vector
with all entries zero except for an entry of unity in the row representing the edge $v_1v_2$. This gives  a
set of linear equations with coefficients in
$\bQ(p)$ and we can select a solution  for which all coordinates of the  velocities $u_i$ lie in $\bQ(p)$.

We  show first that $u$ does not extend to a flex of
$(G^+,p^+)$.

Suppose by way of contradiction that $u^+=(u,u_v)$ is an extension of $u$ to a flex of $(G^+,p^+)$
with component $u_v$ acting at $p_v$. Introducing the notation $p_{i,j}=p_i-p_j$, $p_{v,i}=p_v-p_i$
and similarly $u_{i,j}=u_i-u_j$, $u_{v,i}=u_v-u_i$ the flex $u_v$ satisfies four equations
\[
p_{v,i}.u_{v,i}=0,\quad 1 \leq i \leq 3, \quad  u_v.N(p_v)=0,
\]
where $N(p_v)$ is the normal to the surface $\M$ at $p_v$ given by
\[
N(p_v)= (\nabla m)(p_v) = (\partial m/\partial x, \partial m/\partial y, \partial m/\partial z)|_{p_v}.
\]
Introducing the coordinate notation $(p^{x}_{v,1}, p^{y}_{v,1}, p^{z}_{v,1})$
for $p_{v,1}$ these four equations for the three
components of $u_{v,1}$ have a consistent solution if and only if $\det(D)=0$, where
\[
D=\begin{bmatrix}
p_{v,1}^x&p_{v,1}^y&p_{v,1}^z&0\\
p_{2,1}^x&p_{2,1}^y&p_{2,1}^z&-u_{2,1}.p_{v,2}\\
p_{3,1}^x&p_{3,1}^y&p_{3,1}^z&-u_{3,1}.p_{v,3}\\
N(p_v)^x&N(p_v)^y&N(p_v)^z&u_1.N(p_v)
\end{bmatrix}.
\]
Let $P_v =(x,y,z)$ be the vector of indeterminates corresponding to $p_v$, let $P_{v,i}=P_v-p_i,i=1,2,3$, and let
\[
D(P_v)=D(x,y,z)=\begin{bmatrix}
P_{v,1}^x&P_{v,1}^y&P_{v,1}^z&0\\
P_{2,1}^x&P_{2,1}^y&P_{2,1}^z&-u_{2,1}.p_{v,2}\\
P_{3,1}^x&P_{3,1}^y&P_{3,1}^z&-u_{3,1}.p_{v,3}\\
N(P_v)^x&N(P_v)^y&N(P_v)^z&u_1.N(P_v)
\end{bmatrix}.
\]
Then the polynomial $\det(D(P_v))$ lies in the ring $\bQ(p)[P_v]$.
Since $$0=\det(D)=\det(D(P_v))|_{p_v}$$ the polynomial $\det(D(P_v))$
evaluates to zero under the substitution $P_v=p_v$. Since $p_v$ is  generic on $\M$ this implies that
$\det(D(P_v))$ is in the ideal of $\bQ(p)[P_v]$
generated by the surface polynomial $m(x,y,z)$. Thus $\det(D(P_v))=h(P_v)m(P_v)$ for some
polynomial
$h(P_v)$ in $\bQ(p)[P_v]$.

Since $\det(D(P_v)) = h(P_v)m(P_v)$ and $\nabla(m(P_v))=N(P_v)$ we have
\[\nabla(\det(D(P_v)))=h(P_v)N(P_v)+\nabla(h(P_v))m(P_v)
\]
and so
\[
\nabla(\det(D(P_v)))|_{p_v^\prime}=h(p_v^\prime)N(p_v^\prime)
\]
for any point $p_v^\prime$ satisfying
$m(p_v^\prime)=0$. This implies
$a.\nabla(\det(D(P_v)))|_{p_v^\prime}=0$ for any $a \in \bR^3$ satisfying
$a.N(p_v^\prime)=0$ and
any point $p_v^\prime$ satisfying $m(p_v^\prime)=0$. We consider $p_v^\prime=p_1$ which satisfies this property.

We have $u_1.N(p_1)=0$. Also, since the first row of the matrix $D(P_v)|_{p_1}$ is zero we get a non-zero
contribution to $\nabla(\det(D(P_v)))|_{p_1}$ only from the action of the $\nabla$ operator on the first
row of $D(P_v)$.
Thus, in vector form,
$\nabla(\det(D(P_v)))|_{p_1}$ is the determinant of the matrix
\[
\begin{bmatrix}
i&j&k&0\\
p_{2,1}^x&p_{2,1}^y&p_{2,1}^z&-u_{2,1}.p_{1,2}\\
p_{3,1}^x&p_{3,1}^y&p_{3,1}^z&-u_{3,1}.p_{1,3}\\
N(p_1)^x&N(p_1)^y&N(p_1)^z&u_1.N(p_1)
\end{bmatrix}.
\]
Expanding the determinant along the final column gives
\[
\nabla(\det(D(P_v)))|_{p_1}=((p_{2,1}.u_{2,1})p_{3,1} \times N(p_1)-(p_{3,1}.u_{3,1})p_{2,1} \times N(p_1))
\]
and so from the above
\[a.((p_{2,1}.u_{2,1})p_{3,1} \times N(p_1)-
   (p_{3,1}.u_{3,1})p_{2,1} \times N(p_1))=0
\]
for all $a$ with the property that $a.N(p_1)=0$.

The vector $a=N(p_1) \times (p_{2,1} \times N(p_1))$ satisfies $a.N(p_1)=0$ and
$a.\nabla(\det(D(P_v)))|_{p_1}=0$ gives the
condition $(p_{2,1}.u_{2,1})b=0$ where
\[
b=(N(p_1) \times (p_{2,1} \times N(p_1))).(p_{3,1}\times N(p_1))=N(p_1).(p_{3,1} \times p_{2,1}).
\]
We have $b \neq 0$ because the condition that $N(p_1).(p_{3,1} \times p_{2,1})=0$ for all $p_2$, $p_3$ on $\M$ contradicts 
the smoothness requirement that $p_{3,1} \times p_{2,1}$ becomes parallel to $N(p_1)$ for $p_2$ and $p_3$ close to $p_1$. Thus $p_{2,1}.u_{2,1}=0$ which is
contrary to our original choice of $u$
and so we conclude that $u$ does not extend to a flex of $(G^+,p^+)$.

On the other hand, suppose that a flex $u=(u_1,\dots , u_n)$ of $(G\backslash v_1v_2,p)$
on $\M$ does extend to
a flex $(u,u_v)$ of $(G^+,p^+)$ on $\M$. Then $u_v$ is the solution of the three equations
\[
u_v.(p_v-p_1)=u_1.(p_v-p_1), u_v.(p_v-p_2)=u_2.(p_v-p_2) \mbox{ and } u_v.N(p_v)=0
 \]
and the solution is
unique because
$(p_v-p_1)\times (p_v-p_2).N(p_v) \neq 0$ for generic $p_v,p_1,p_2$ for the same reason given
above that $b \neq 0$. Also if $u$ is zero then $(u,u_v)$ is zero
and so every flex in the nullspace of $R_M(G^+,p^+)$ is the extension of a flex of $R_\M(G\backslash v_1v_2,p)$.

Finally, consider the matrix  $R^\prime = R_M(G\backslash v_1v_2,p)$  of size $m^\prime \times n^\prime$. Since $G$ is independent we have
$m^\prime+1 \leq n^\prime$ and $\rank(R^\prime)=m^\prime$. For the matrix $R=R_\M(G^+,p^{+})$  of size $m \times n$ we have $m=m^\prime+4$
and $n = n^\prime+3$ and so $m \leq n$ and $n-m = n^\prime - m^\prime - 1$. Every flex of
$R^\prime$ either does not extend to a flex of $R$ or extends to a unique flex of $R$ and every
flex of $R$ is the extension of some flex of $R^\prime$. Thus
$|\nulty(R^\prime)| > |\nulty(R)|$. By
Lemma \ref{lem1} the rank of $R$ is $m$ which means that $G^+$ is independent on $\M$, as required.

\begin{lem} \label{lem1}
Let $R$ be an $m \times n$ matrix with $m \leq n$ and $R^\prime$ an $m^\prime \times n^\prime$ matrix
with
$m^\prime \leq n^\prime$ and $\rank(R^\prime )=m^\prime$. If $n-m = n^\prime - m^\prime - 1$ and
$|\nulty(R)| < |\nulty(R^\prime)|$ then $\rank(R)=m$.
\end{lem}

\begin{proof}
$\rank(R)= n - |\nulty(R)| \geq n-|\nulty(R^\prime)|+1$ and
$|\nulty(R^\prime)| = n^\prime-m^\prime$ so $\rank(R) \geq m$.
\end{proof}

\section{Minimal rigidity on type $1$ Surfaces}
\label{theoremsec}

There is a final independence preserving move that we need for the proof of the main result. Recall that
if $G$ and $H$ are graphs with vertices $g \in G, h \in H$ then the \emph{edge joining move}
combines $G$ and $H$ by adding the edge $gh$.

%For the plane or the sphere, if $G$ and $H$ have minimally rigid generic realisations then $R_\M(G,p)$ and
%$R_\M(H,q)$ have $3$-dimensional nullspaces. Thus  $R_\M(G\cup H,(p,q))$ has a $6$-dimensional nullspace and adding $gh$ adds at most one to
%the rank.
%Similarly, infinitesimal rigidity is
%not preserved in the case of the cylinder since the nullspace is at least $3$-dimensional.
%However, we have the following.

\begin{lem}\label{edgejoin}
Let  $\M$ be an irreducible  surface of type $1$. Let $(G,p)$ and $(H,q)$ be
minimally infinitesimally rigid on $\M$, and let $G'$ be an edge join of $G$ and $H$ through an edge $gh$. If
 $|V(G)|$ and $|V(H)|$ are greater than $4$ and $(p_g,p_h)$ is generic on $\M$ then $(G', (p,q))$ is generically minimally infinitesimally rigid on $\M$.
\end{lem}

\begin{proof}
Consider the block matrix form
\begin{equation*}
R_{\M}(G',(p,q))=
 \begin{bmatrix}
R_\M(G,p) & 0\\
* & *\\
0 & R_\M(H,q)
 \end{bmatrix}.
\end{equation*}
By the hypotheses the nullspaces of the rigidity matrices $R_\M(G,p)$ and  $R_\M(H,q)$
are one-dimensional. Let $u=(u_p,u_q)$ be an infinitesimal flex of the edge-joined framework and let $gh$ be the joining edge for $G'$.
Subtracting a tangential rigid motion infinitesimal flex we may assume that $u_p$ assigns a zero velocity to
the framework joint $p_g$. Since $u_p$ is a flex of $(G,p)$ on $\M$ it follows that $u_p=0.$ Also in view of the generic row in $R_{\M}(G',(p,q))$ for the joining edge it follows that $u_q$ assigns a velocity to the framework vertex $q_h$ which is linearly independent from the one-dimensional space of velocity vectors of $q_h$ obtained from infinitesimal flexes of $(H,q)$. It follows that this velocity on $q_h$ is zero and hence that $u_q$ is zero.
Thus $u=0$ and the nullspace of $R_{\M}(G',(p,q))$ has dimension one, as desired.
\end{proof}

We now arrive at the proof of our main result, Theorem \ref{conetorustheorem}. For the reader's convenience we first re-state the theorem.

\begin{thm}\label{conetorustheorem2}
Let $G=(V,E)$ be a simple graph and let $\M$ be an irreducible  surface
of type $1$. Then a
generic framework $(G,p)$  on $\M$ is isostatic if and only
$G$ is $K_1, K_2, K_3, K_4$ or is $(2,1)$-tight.
\end{thm}

\begin{proof} That the underlying graph of an isostatic framework on $\M$ is $(2,1)$-tight or is a small complete graph follows from
Theorem \ref{necessity}.
For the sufficiency direction one can check  that the minimal graph $K_5\backslash e$ in the inductive characterisation of $(2,1)$-tight graphs is isostatic on
$\M$. The sufficiency of $(2,1)$-tightness now follows from Theorem \ref{t:21characterisation} if minimal generic rigidity is preserved by Henneberg 1 and 2 moves, the vertex-to-$K_4$ move, the vertex-to-$4$-cycle move and the edge joining move. This is the content of Lemma \ref{Hen1}, Lemma \ref{hen2independent}, Corollary \ref{vertex2k4},
Lemma \ref{l:Vertexto4cycle} and Lemma \ref{edgejoin}.
\end{proof}

Note that we could also have used Theorem \ref{21}, applying Lemma \ref{Vertexsplitprop}, to prove the theorem.

\section{Extensions}

We finish by noting some further natural considerations for frameworks constrained to surfaces.

The assumption, in Theorem \ref{conetorustheorem}, that $\M$ is irreducible merits two comments. Firstly, it is required to avoid to surfaces composed as unions of surfaces of differing numbers of internal motions. For example if $\M$ was the union of two cylinders with distinct but parallel axes then consideration of Theorem \ref{unioncylinders} instantly shows that simply being $(2,1)$-tight is not the correct characterisation. Secondly, if $\M$ is reducible but each component is irreducible then we do not expect any great difficulty in extending our results. For example we expect that Theorem \ref{conetorustheorem} is true for concentric cones, torii or elliptical cylinders.

The usual two-dimensional torus embedded in $3$-dimensional space has freedom type $1$ and so  isostatic frameworks on this surface are characterised as in the previous theorem.
When the torus is realised in $\bR^4$ one may also consider the \emph{Clifford torus} $\T$, that is, the real algebraic variety and smooth manifold defined by the polynomial equations $x^2+y^2=1$ and $z^2+w^2=1$.
The definition of type (freedom number) given in Definition \ref{d:surfacetype} extends without change to an algebraic surface $\M$ in $\bR^d$.

\begin{defn}\label{d:surfacetype2}
An  embedded manifold $\M$ in $\bR^d$ is of type $k$ if $\dim \ker R_{\M}(K_n,q)\geq k$ for all frameworks $(K_n,p)$ on $\M$, for $n= 2,3, \dots$, and $k$ is the largest such integer.
\end{defn}

In particular the
Clifford torus has freedom type $2$.
The rigidity analysis in this setting
requires us to adapt the definition of the rigidity matrix. The details are
similar to those in Definition \ref{rigiditymatrixdef} with the following changes. There are now $4$ columns per vertex and $2$ rows per vertex
where the rows for vertex $i$ (and corresponding framework point $(x_i,y_i,z_i,w_i)$) are zero except in the $4$-tuple corresponding to
$i$ where the entries in the first row are $x_i,y_i,0,0$ and the second are $0,0,z_i,w_i$.

On the other hand take the product of a circle and an ellipse. This is the algebraic variety $\S$ defined by, say, $x^2+y^2=1$ and $z^2+w^2/2=1$ in
$\bR^4$. $\S$ admits exactly $1$ trivial motion. Adapting the methods 
of the last section would lead to the $(2,2)$-tight and the $(2,1)$-tight characterisations of frameworks on $\T$ and $\S$ respectively.

It is natural to seek a similar characterisation of our main result in the case of frameworks with vertices constrained to an irreducible surface of type $0$.  There are a variety of such surfaces that a characterisation could apply to including an elliptical cone, a mobius strip, a
hyperboloid and a hyperbolic paraboloid. Note that the graphs of rigid frameworks need not be connected in this setting.
As a starting point \cite[Proposition $3.4$]{NOP} gives the necessity of the graph being simple and $(2,0)$-tight and we expect that the rigidity preservation methods in this paper will be useful in deriving a characterisation. However, there are immediate additional complications to establishing sufficient conditions, not least since any simple $(2,0)$-tight graph containing a subgraph isomorphic to $K_5$ has a dependent rigidity matrix on any surface. This can easily be seen by noting that $K_5$ is not $(3,6)$-sparse and hence is dependent in $\bR^3$. We also remark that $(2,0)$-tight graphs may be $4$-regular so additional degree 4 operations seem to be necessary (such as $X$-replacement, see \cite{nix-ros,T&W}). This fact, together with the fact that an inductive scheme would have to avoid creating $K_5$ subgraphs, and the fact that there are many simple $(2,0)$-tight graphs (even on small vertex sets) that cannot be generated using the operations in this paper, with $X$-replacement, all suggest that the analogue of Theorem \ref{t:21characterisation} for simple $(2,0)$-tight graphs, without $K_5$ subgraphs, will be significantly more challenging to establish.

\medskip

{\bf Acknowledgement.}
We would like to thank Bill Jackson, for discussions relating to the Henneberg $2$ move on manifolds.

\end{document}